\documentclass[reqno]{amsart}
\usepackage{amsmath}%
\usepackage{amsfonts}%
\usepackage{amssymb}%
\usepackage{graphicx}
\usepackage{mathrsfs}
\usepackage{hyperref}
\usepackage{fullpage}
\def\0{\emptyset}

\def\n{\noindent}

\newtheorem{theorem}{Theorem}
\newtheorem{lemma}[theorem]{Lemma}

\newtheorem{claim}[theorem]{Claim}

\newtheorem{conjecture}[theorem]{Conjecture}

\newtheorem{problem}[theorem]{Problem}
\usepackage{graphicx}

\begin{document}
\title{Vertex degree sums for perfect matchings in 3-uniform hypergraphs }
\thanks{
Yi Zhang and Mei Lu are supported by the National Natural Science Foundation of China (Grants 61373019 and 11771247).
Yi Zhao is partially supported by NSF grants DMS-1400073 and DMS-1700622.}

\author{Yi Zhang}
\address{ Department of Mathematical
Sciences, Tsinghua University, Beijing, 100084}
\email{shouwangmm@sina.com}

\author{Yi Zhao}
\address
{Department of Mathematics and Statistics, Georgia State University, Atlanta, GA 30303}
\email{yzhao6@gsu.edu}

\author{Mei Lu}
\address{Department of Mathematical
Sciences, Tsinghua University, Beijing, 100084}
\email{mlu@math.tsinghua.edu.cn}

\date{\today}


\begin{abstract}
We determine the minimum degree sum of two adjacent vertices that ensures a perfect matching in a 3-graph without isolated vertex. More precisely, suppose that $H$ is a 3-uniform hypergraph whose order $n$ is sufficiently large and divisible by $3$. If $H$ contains no isolated vertex and $\deg(u)+\deg(v) > \frac{2}{3}n^2-\frac{8}{3}n+2$ for any two vertices $u$ and $v$ that are contained in some edge of $H$, then $H$ contains a perfect matching. This bound is tight.
\end{abstract}

\maketitle

\section{Introduction}

A $k$-uniform hypergraph (in short, $k$-graph) $H$ is a pair $(V,E)$, where $V := V (H)$ is a finite set
of vertices and $E := E(H)$ is a family of $k$-element subsets of $V$. A matching of size $s$ in $H$ is a family of $s$ pairwise disjoint edges of $H$. If the matching covers all the vertices of $H$, then we call it a \emph{perfect matching}.
Given a set $S \subseteq V$, the \emph{degree} $\deg_{H}(S)$ of $S$ is the number of the edges of $H$ containing $S$. We simply write $\deg(S)$ when $H$ is obvious from the context.

Given integers $\ell<k\le n$ such that $k$ divides $n$, we define the minimum {\em $\ell$-degree threshold} $m_\ell(k,n)$ as the smallest integer $m$ such that every $k$-graph $H$ on $n$ vertices with $\delta_\ell(H) \geq m$ contains a perfect matching. In recent years the problem of determining $m_\ell(k,n)$ has received much attention, see, \emph{e.g.}, \cite{Alon,Han,Han3,Kha1, Kha2,Kuhn1,Kuhn4,Kuhn3,Mar,Pik,Rod2,Rod1,Rod3,Tre,TrZh13,Town2}. In particular, R\"{o}dl,  Ruci\'{n}ski, and Szemer\'{e}di \cite{Rod3} determined $m_{k-1}(k, n)$ for all $k\ge 3$ and sufficiently large $n$.
Treglown and Zhao \cite{Tre, TrZh13} determined $m_\ell(k,n)$ for all $\ell \geq k/2$ and sufficiently large $n$.
For more Dirac-type results on hypergraphs, we refer readers to surveys \cite{RoRu-s, Zhao}.



In this paper we consider vertex degrees in 3-graphs. H\`{a}n, Person and Schacht \cite{Han} showed that
\begin{equation}
\label{eq:Nu4}
m_1(3,n) =  \left(\frac{5}{9}+o(1)\right)\binom{n}{2}.
\end{equation}
K\"{u}hn, Osthus and Treglown \cite{Kuhn2} and independently Khan \cite{Kha1} later proved that $m_1(3,n) = \binom{n-1}{2} - \binom{2n/3}{2}+1$ for sufficiently large $n$.



Motivated by the relation between Dirac's condition and Ore's condition for Hamilton cycles, Tang and Yan \cite{Tang} studied the degree sum of two $(k-1)$-sets that guarantees a tight Hamilton cycle in $k$-graphs. Zhang and Lu \cite{Yi1} studied the degree sum of two $(k-1)$-sets that guarantees a perfect matching in $k$-uniform hypergraphs.

It is more natural to consider the degree sum of two vertices that guarantees a perfect matching in hypergraphs.
For two distinct vertices $u, v$ in a hypergraph, we call $u,v$ \emph{adjacent} if there exists an edge containing both of them. The following are three possible ways of defining the minimum degree sum of 3-graphs.
Let $\sigma_2(H) = \min \{\deg(u)+\deg(v): u, v \in V(H)\}$, $\sigma_2'(H)=\min \{\deg(u)+\deg(v): u \,\, \text{and} \,\, v \,\, \text{are} \,\, \text{adjacent} \}$ and $\sigma_2''(H) = \min \{\deg(u)+\deg(v): u \,\, \text{and} \,\, v \,\, \text{are} \,\, \text{not} \,\,  \text{adjacent} \}$.

The parameter $\sigma_2$ is closely related to the Dirac threshold $m_1(3, n)$. We can prove that when $n$ is divisible by 3 and sufficiently large, every 3-graph $H$ on $n$ vertices with $\sigma_2(H) \geq  2 (\binom{n-1}{2} - \binom{2n/3}{2})+1$ contains a perfect matching. Indeed,  such $H$ contains at most one vertex $u$ with $\deg (u) \leq \binom{n-1}{2} - \binom{2n/3}{2}$. If $\deg (u) \le (5/9 -\varepsilon) \binom{n}2$ for some  $\varepsilon > 0 $, then we choose an edge containing $u$ and find a perfect matching in the remaining 3-graph by \eqref{eq:Nu4} immediately. Otherwise, $\delta_1(H) \ge (5/9 -\varepsilon) \binom{n}2$. We can prove that $H$ contains a perfect matching by following the same process as in \cite{Kuhn2}.\footnote{In fact, due to the absorbing method, we only need to verify the extremal case.}


On the other hand, no condition on $\sigma_2''$ alone guarantees a perfect matching. In fact, let $H$ be the 3-graph whose edge set consists of all triples that contain a fixed vertex. This $H$ contains no two disjoint edges even though it satisfies all conditions on $\sigma_2''$ (because any two vertices of $H$ are adjacent).

Therefore we focus on $\sigma'_2$. More precisely, we determine the largest $\sigma_2'(H)$ among all 3-graphs $H$ of order $n$ without isolated vertex such that $H$ contains no perfect matching. (Trivially $H$ contains no perfect matching if it contains an isolated vertex.)
Let us define a 3-graph $H^*$, whose vertex set is partitioned into two vertex classes $S$ and $T$ of size $n/3+1$ and $2n/3-1$, respectively, and whose edge set consists of all the triples containing at least two vertices of $T$. For any two vertices $u \in T$ and $v \in S$,
\[
\deg(u) = \binom{2n/3-2}{2}+\left(\frac{n}{3}+1\right)\left(\frac{2n}{3}-2\right)> \binom{2n/3-1}{2} = \deg(v).
 \]
Hence $\sigma_2'(H^*) = \binom{2n/3-2}{2}+(n/3+1)(2n/3-2)+\binom{2n/3-1}{2}=2n^2/3 -{8}n/{3} +2$. Obviously, $H^*$ contains no perfect matching. The following is our main result.

\begin{theorem}\label{the1}
There exists $n_0 \in \mathbb{N}$ such that the following holds for all integers $n\ge n_0$ that are divisible by $3$. Let $H$ be a $3$-graph of order $n \geq n_0$ without isolated vertex. If $\sigma_2'(H) > \sigma'_2(H^*)= \frac{2}{3}n^2-\frac{8}{3}n+2$, then $H$ contains a perfect matching.
\end{theorem}

Theorem \ref{the1} actually follows from the following stability result.

\begin{theorem} \label{theorem6}
There exist $\varepsilon>0$ and $n_0 \in \mathbb{N}$ such that the following holds for all integers $n\ge n_0$ that are divisible by $3$.  Suppose that $H$ is a $3$-graph of order $n \geq n_0$ without isolated vertex and $\sigma_2'(H) > 2n^2/3-\varepsilon n^2$, then $H \subseteq H^*$ or $H$ contains a perfect matching.
\end{theorem}

Indeed, if $\sigma_2'(H) > 2n^2/3 -{8}n/{3} +2$, then $H \nsubseteq H^*$ and by Theorem \ref{theorem6}, $H$ contains a perfect matching. Furthermore, Theorem \ref{theorem6} implies that $H^*$ is the unique extremal 3-graph for Theorem \ref{the1} because all proper subgraphs $H$ of $H^*$  satisfy $\sigma'_2(H) < \sigma'_2(H^*)$.



This paper is organized as follows. In Section 2, we provide preliminary results and an outline of our proof. We prove an important lemma in Section 3 and we complete the proof of Theorem \ref{theorem6} in Section 4. Section  5 contains concluding remarks and open problems.

\vskip.2cm

\n{\bf Notation:}
Given vertices $v_1, \dots, v_t$, we often write $v_1 \cdots v_t$ for $\{v_1, \dots, v_t\}$.
The neighborhood $ N(u,v)$ is the set of the vertices $w$ such that $uvw \in E(H)$. Let $V_1,V_2,V_3$ be three vertex subsets of $V(H)$, we say that an edge $e\in E(H)$ is of type $V_1V_2V_3$ if $e=\{v_1,v_2,v_3\}$ such that $v_1 \in V_1$, $v_2 \in V_2$ and $v_3\in V_3$.

Given a vertex $v\in V(H)$ and a set $A \subseteq V(H)$, we define the \emph{link} $L_{v}(A)$ to be the set of all pairs $uw$ such that $u, w\in A$ and $uvw\in E(H)$. When $A$ and $B$ are two disjoint sets of $V(H)$, we define $L_{v}(A,B)$ as the set of all pairs $uw$ such that $u\in A$, $w\in B$ and $uvw\in E(H)$.

We write $0 < a_1 \ll a_2 \ll a_3 $ if we can choose the constants $a_1, a_2, a_3$ from right to left. More precisely there are increasing functions $f$ and $g$ such that given $a_3$, whenever we choose some $a_2\leq f(a_3)$ and $a_1 \leq g(a_2)$, all calculations needed in our proof are valid.

\vskip.2cm

\section{Preliminaries and proof outline }

\vskip.2cm

We will need small constants
\begin{align*}
0 <  \varepsilon  \ll \eta \ll \gamma \ll \gamma' \ll \rho \ll \tau \ll 1.
\end{align*}
Suppose $H$ is a 3-graph such that $\sigma_2'(H) > 2n^2/3-\varepsilon n^2$. Let $W = \{v\in V(H): \deg(v) \leq n^2/3-\varepsilon n^2/2\}$, $U = V\setminus W$. If $W= \emptyset$, then \eqref{eq:Nu4} implies that $H$ contains a perfect matching. We thus assume that $|W|\ge 1$.
Any two vertices of $W$ are not adjacent -- otherwise $\sigma'_2(H) \leq 2n^2/3-\varepsilon n^2$, a contradiction.  If $|W| \geq n/3+1$, then $H \subseteq H^*$ and we are done. We thus assume $|W| \leq n/3$ for the rest of the proof.

Our proof will use the following claim.

\begin{claim}\label{claim2}
If $|W| \geq n/4$, then every vertex of $U$ is adjacent to some vertex of $W$.
\end{claim}
\begin{proof} To the contrary, assume that some vertex $u_0\in U$ is not adjacent to any vertex in $W$. Then we have $\deg(u_0) \leq \binom{|U|-1}{2}=\binom{n-|W|-1}{2}$. Since $|W| \geq n/4$ and $n$ is sufficiently large,

\begin{align*}
\deg(u_0) \leq \binom{n-n/4-1}{2} = \frac{9}{32}n^2-\frac{9}{8}n+1 < \frac{n^2}{3} -\frac{\varepsilon}2 n^2,
\end{align*}
which contradicts the definition of $U$.
\end{proof}

By Claim \ref{claim2}, when $|W| \geq \frac{n}{4}$, we have $\deg(u) \geq (2n^2/3-\varepsilon n^2)-\binom{n-|W|}{2}$
for every $u\in U$. This is stronger than the bound given by the definition of $U$ because
\[
\left(\frac{2}{3}n^2 - \varepsilon n^2\right)-\binom{n-|W|}{2} \ge \left(\frac{2}{3}n^2 - \varepsilon n^2\right)-\binom{n- n/4}{2} > \frac{n^2}{3}-\frac{\varepsilon}{2} n^2.
\]

\medskip
Our proof consists of two steps.

\noindent{\bf Step 1.}  We prove that $H$ contains a matching that covers all the vertices of $W$.

\begin{lemma} \label{lemma7}
There exist $\varepsilon >0$ and $n_0 \in \mathbb{N}$ such that the following holds. Suppose that $H$ is a $3$-graph of order $n \geq n_0$ without isolated vertex and $\sigma_2'(H) > 2n^2/3-\varepsilon n^2$.  Let $W =\{v \in V(H): \deg(v) \le n^2/3-\varepsilon n^2/2\}$. If $|W| \leq n/3$, then $H$ contains a matching that covers every vertex of $W$.
\end{lemma}

Our approach towards Lemma \ref{lemma7} begins by considering a largest matching $M$ such that every edge of $M$ contains one vertex from $W$ and suppose $|M| < |W|$. If $|W| \leq (1/3-\gamma)n$, then we choose two adjacent vertices, one from $W$ and the other from $V\setminus W$ to derive a contradiction with $\sigma'_2(H)$. If $n/3 \geq |W| > (1/3-\gamma)n  $, we use three unmatched vertices, one from $W$ and two from $V\setminus W$ to derive a contradiction.

\noindent{\bf Step 2.} We show that $H$ contains a perfect matching. 


Because of Lemma \ref{lemma7}, we begin by considering a largest matching $M$ such that $M$ covers every vertex of $W$ and suppose that $|M| < n/3$. After choosing three vertices from $V\setminus V(M)$, we distinguish the cases when $|M| \leq n/3-\eta n$ and when $|M| > n/3-\eta n$ and derive a contradiction by comparing upper and lower bounds for the degree sum of these three vertices. When $|M| > n/3-\eta n$, we need to apply \eqref{eq:Nu4}.

\medskip
In Step 2 we need three simple extremal results. The first lemma is Observation 1.8 of Aharoni and Howard \cite{Aha}.
A $k$-graph $H$ is called \emph{$k$-partite} if $V(H)$ can be partitioned into $V_1,\cdots,V_k$, such that each edge of $H$ meets every $V_i$ in precisely one vertex. If all parts are of the same size  $n$, we call $H$ \emph{$n$-balanced}.
\begin{lemma}\cite{Aha}
\label{lemma1}
Let $F$ be the edge set of an $n$-balanced $k$-partite $k$-graph. If $F$ does not contain $s$ disjoint edges, then $|F| \leq (s-1)n^{k-1}$.
\end{lemma}

The bound in the following lemma is tight because we may let $G_1$ be the empty graph and $G_2=G_3= K_n$.

\begin{lemma}\label{lemma2}
Given two sets $A\subset V$ such that $|A|=3$ and $|V|=n\ge 4$, let $G_1, G_2, G_3$ be three graphs on $V$ such that
no edge of $G_1$ is disjoint from an edge from $G_2$ or $G_3$.
Then $\sum_{i=1}^3\sum_{v\in A} \deg_{G_i}(v) \leq 6(n-1)$.
\end{lemma}
\begin{proof}
Assume $A=\{u_1,u_2,u_3\}$ and let $b= n-3\ge 1$. We need to show that $\sum_{i=1}^3\sum_{j=1}^3 \deg_{G_i}(u_j) \le 6b+12$.

Let $\ell_i$ denote the number of the vertices in $A$ of degree at least 3 in $G_i$. We distinguish the following two cases:

\noindent{\bf Case 1: } $\ell_1 \geq 1$.

If $\ell_1 \geq 2$, say, $\deg_{G_1}(u_j)\geq 3$ for $j=1,2$, then $E(G_i) \subseteq \{ u_1u_2 \}$ for $i=2,3$ -- otherwise we can find two disjoint edges, one from $G_1$ and the other from $G_2$ or $G_3$. Therefore, $\sum_{j=1}^3 \deg_{G_i}(u_j) \leq 2$ for $i=2,3$. Moreover, $\sum_{j=1}^3 \deg_{G_1}(u_j) \leq 3b+6$. We have $\sum_{i=1}^3\sum_{j=1}^3 \deg_{G_i}(u_j) \leq 3b+10 < 6b+12$.

If $\ell_1 = 1$, say, $\deg_{G_1}(u_1)\geq 3$, then $G_i$ is a star centered at $u_1$ for $i=2,3$ -- otherwise one edge of $G_1$ must be disjoint from one edge of $G_2$ or $G_3$. In this case we have $\sum_{j=1}^3 \deg_{G_1}(u_j) \leq b+2+4$ and $\sum_{j=1}^3 \deg_{G_i}(u_j) \leq b+4$ for $i=2,3$. Therefore, $\sum_{i=1}^3\sum_{j=1}^3 \deg_{G_i}(u_j) \leq 3b+14 < 6b+12$ as $b \geq 1$.

\noindent{\bf Case 2: } $\ell_1 = 0$.

If $\ell_i=3$ for some $i \in \{2,3\}$, then $E(G_1) = \emptyset$. In this case $\sum_{i=1}^3\sum_{j=1}^3 \deg_{G_i}(u_j) \leq 2(3b+6) \leq 6b+12$.

Suppose $\ell_2, \ell_3\le 2$ and $\ell_2=2$ or $\ell_3=2$. Without loss of generality, assume $\ell_2=2$ and $\deg_{G_2}(u_j) \geq 3$ for $j=1,2$. Then $E(G_1) \subseteq \{ u_1u_2 \}$. In this case $\sum_{j=1}^3 \deg_{G_1}(u_j) \leq 2$ and $\sum_{j=1}^3 \deg_{G_i}(u_j) \leq 2b+4+2$ for $i=2,3$. Hence $\sum_{i=1}^3\sum_{j=1}^3 \deg_{G_i}(u_j) \leq 4b+14 \leq 6b+12$ as $b \geq1$.

Assume $\ell_2$, $\ell_3\le 1$ and $\ell_2=1$ or $\ell_3=1$. Without loss of generality, assume $\ell_2=1$ and $\deg_{G_2}(u_1) \geq3$. Then $G_1$ is a star centered at $u_1$. We have $\sum_{j=1}^3 \deg_{G_1}(u_j) \leq 4$ and $\sum_{j=1}^3 \deg_{G_i}(u_j) \leq b+2+4$ for $i=2,3$. So $\sum_{i=1}^3\sum_{j=1}^3 \deg_{G_i}(u_j) \leq 2b+16 \leq 6b+12$ as $b \geq 1$.

Suppose $\ell_2$, $\ell_3 = 0$. In this case $\sum_{j=1}^3 \deg_{G_i}(u_j) \leq 6$ for $i=1,2,3$. Therefore,  $\sum_{i=1}^3\sum_{j=1}^3 \deg_{G_i}(u_j) \leq 18 \leq 6b+12$ as $b \geq 1$.
\end{proof}

The bound in the following lemma is tight because we may let $G_1=G_2=G_3$ be a star of order $n$ centered at a vertex of $A$.

\begin{lemma}\label{lemma3}
Given two sets $A\subset V$ such that $|A|=3$ and $|V|=n\ge 5$, let $G_1, G_2, G_3$ be three graphs on $V$ such that
no edge of $G_i$ is disjoint from an edge from $G_j$ for any $i\ne j$.
Then $\sum_{i=1}^3\sum_{v\in A} \deg_{G_i}(v) \leq 3(n+1)$.
\end{lemma}
\begin{proof} Assume $A=\{u_1,u_2,u_3\}$ and let $b= n-3\ge 2$. We need to show that $\sum_{i=1}^3\sum_{j=1}^3 \deg_{G_i}(u_j) \le 3b+12$.

Let $\ell_i$ denote the number of the vertices in $A$ of degree at least 3 in $G_i$. We distinguish the following two cases:

\noindent{\bf Case 1:} $\ell_i \geq 1$ for some $i \in [3]$.

Without loss of generality, $\ell_1\geq1$ and $\deg_{G_1}(u_1) \geq 3$. If $\deg_{G_1}(u_2) \geq 3$ or $\deg_{G_1}(u_3) \geq 3$, say, $\deg_{G_1}(u_2) \geq 3$, then $E(G_i) \subseteq \{ u_1u_2 \}$ for $i=2,3$ -- otherwise we can find two disjoint edges $e_1$ and $e_2$ from two distinct graphs of $G_1,G_2,G_3$. In this case $\sum_{j=1}^3 \deg_{G_1}(u_j) \leq 3b+6$ and $\sum_{j=1}^3 \deg_{G_i}(u_j) \leq 2$ for $i=2,3$, which implies that $\sum_{i=1}^3\sum_{j=1}^3 \deg_{G_i}(u_j) \leq 3b+10$.

Assume $\deg_{G_1}(u_j) \leq 2$ for $j=2,3$. We know that $G_i$, $i=2,3$ is a star centered at $u_1$ -- otherwise one edge of $G_1$ must be disjoint from one edge of $G_i$, $i \in \{2,3\}$. If $\deg_{G_2}(u_1) \geq 3$ or $\deg_{G_3}(u_1) \geq 3$, then $G_1$ is also a star centered at $u_1$. In this case $\sum_{j=1}^3 \deg_{G_i}(u_j) \leq b+4$ for $i \in [3]$, so $\sum_{i=1}^3\sum_{j=1}^3 \deg_{G_i}(u_j) \leq 3b+12$. Otherwise $\deg_{G_i}(u_1) \leq 2$ for $i=2,3$, hence $\sum_{j=1}^3 \deg_{G_i}(u_j) \leq 4$ for $i=2,3$. Since $\sum_{j=1}^3 \deg_{G_1}(u_j) \leq b+6$, we have $\sum_{i=1}^3\sum_{j=1}^3 \deg_{G_i}(u_j) \leq b+14 \leq 3b+12$.

\noindent{\bf Case 2:} $\ell_i=0$ for $i\in[3]$.

In this case $\sum_{j=1}^3 \deg_{G_i}(u_j) \leq 6$ for $i=1,2,3$.  Hence $\sum_{i=1}^3\sum_{j=1}^3 \deg_{G_i}(u_j) \leq 18 \leq 3b+12$ as $b \geq 2$.
\end{proof}


\section{Proof of Lemma \ref{lemma7}}

Choose a largest matching of $H$, denoted by $M$, such that every edge of $M$ is of type $UUW$. To the contrary, assume that $|M| \leq |W|-1$.
Let $U_1= V(M) \cap U$, $U_2 = U\setminus U_1$, $W_1 = V(M) \cap W$, and $W_2=W\setminus W_1$. Then $|U_1|=2|M|$, and $|U_2| = n-|W|-2|M|$. We distinguish the following two cases.

\noindent{\bf Case 1:} $0 < |W| \leq (\frac{1}{3}-\gamma)n$. 

We further distinguish the following two sub-cases:

\smallskip
\noindent{\bf Case 1.1:} A vertex $v_0\in W_2$ is adjacent to a vertex $u_0\in U_2$.

Let $M' = \{e \in M: \exists \, u'\in e, |N(v_0,u')\cap U_2| \geq 3\}$. Assume $\{u_1,u_2,v_1\} \in M'$ such that $u_1,u_2 \in U_1$, $v_1 \in W_1$, and $|N(v_0,u_1)\cap U_2|\geq 3$. We claim that
\begin{equation}
\label{eq:Nu0}
N(u_0,v_1)\cap (U_2 \cup \{u_2\}) = \emptyset.
\end{equation}
Indeed, if $\{u_0, v_1, u_3\} \in E(H)$ for some $u_3\in U_2$, then we can find $u_4\in U_2\setminus \{u_0, u_3\}$ such that $\{v_0, u_1, u_4\}\in E(H)$. Replacing $\{u_1,u_2,v_1\}$ by $\{u_0, v_1, u_3\}$ and $\{v_0, u_1, u_4\}$
gives a larger matching than $M$, a contradiction. The case when $\{u_0, v_1, u_2\}\in E(H)$ is similar.

By the definition of $M'$, there are at most $2(|U_1|-2|M'|)$ edges containing $v_0$ with one vertex in $U_1\setminus V(M')$ and one vertex in $U_2$.  This implies that
\[
\deg(v_0) \leq \binom{|U_1|}{2}+2|M'||U_2|+2(|U_1|-2|M'|) =  \binom{|U_1|}{2} + 2|U_1| + |M'| (2|U_2|-4).
\]
By \eqref{eq:Nu0}, there are at most $|U_1||W_1|-|M'|$ edges consisting of $u_0$, one vertex in $U_1$, and one vertex in $W_1$, and at most $(|U_2|-1)(|W_1|-|M'|)$ edges consisting of $u_0$, one vertex in $U_2$, and one vertex in $W_1$. Therefore,
\begin{align*}
\deg(u_0) &\leq \binom{|U|-1}{2}+|U_1||W_2|+|U_1||W_1|-|M'|+(|U_2|-1)(|W_1|-|M'|)\\
			&= \binom{|U|-1}{2}+|U_1||W| + (|U_2|-1) |W_1| - |U_2| |M'|,
\end{align*}
and consequently,
\[
\deg(v_0)+\deg(u_0) \le \binom{|U_1|}{2} + 2|U_1| + \binom{|U|-1}{2}+|U_1||W| + (|U_2|-1) |W_1| + |M'| (|U_2|-4).
\]
Since $|W|\le (\frac{1}{3}-\gamma)n$, we have $|U_2| > 3\gamma n>4$. As $|M'| \le |M| = |W_1|= \frac{|U_1|}{2}$, it follows that
\begin{align*}
\deg(v_0)+\deg(u_0)&\leq \binom{|U_1|}{2}+2|U_1|+ \binom{|U|-1}{2}+|U_1||W|+(|U_2|-1)\frac{|U_1|}{2}+\frac{|U_1|}{2}(|U_2|-4) \\
&= \left(\binom{|U|}{2}- \binom{|U_2|}{2}\right) +\binom{|U|-1}{2} +\left(|W|-\frac{1}{2}\right)|U_1|\\
&= \left( |U| - 1 \right)^2 -  \binom{|U_2|}{2} +  (2|W|-1)|M|.
\end{align*}
Since $|M| \leq |W|-1$ and $|U_2|\ge n - 3|W| + 2$, we derive that
\begin{align*}
\deg(v_0)+\deg(u_0) &\le (n- |W| -1)^2 - \binom{n- 3|W| + 2}2 + (2|W|-1)(|W|-1)\\
&= \frac{2}{3}n^2-\frac{7}{3}n+\frac{73}{24} - \frac{3}{2}\left( \frac{n}{3}+\frac{7}{6} - |W|\right)^2.
\end{align*}
Since $ |W| \leq (\frac{1}{3}-\gamma)n$, $0 < \varepsilon \ll \gamma$ and $n$ is sufficiently large, we have
\begin{align*}
\deg(v_0)+\deg(u_0) \leq   \frac{2}{3}n^2-\frac{7}{3}n+\frac{73}{24} - \frac{3}{2}\left( \gamma n+\frac{7}{6}\right)^2  < \frac{2}{3}n^2-\varepsilon n^2.
\end{align*}
This contradicts our assumption on $\sigma_2'(H)$ because $v_0$ and $u_0$ are adjacent.

\smallskip
\noindent{\bf Case 1.2:} No vertex in $W_2$ is adjacent to any vertex in $U_2$.

Fix $v_0\in W_2$. Since $v_0$ is not adjacent to any vertex in $U_2$, we have $\deg(v_0) \leq \binom{|U_1|}{2}=\binom{2|M|}{2}$.
Since $v_0$ is not an isolated vertex, there exists a vertex $u_1\in U_1$ that is adjacent to $v_0$. By the assumption, $H$ contains no edge containing $u_1$ with one vertex in $U_2$, one vertex in $W_2$. Thus
$\deg(u_1)\leq \binom{|U|-1}{2}+(|U|-1)|W|-|U_2||W_2|$.
Since $|M| \leq |W|-1$ and $|U|= n - |W|$, it follows that
\begin{align*}
\deg(v_0)+\deg(u_1) &\leq \binom{2(|W|-1)}{2}+\binom{|U|-1}{2}+(|U|-1)|W |-(n-3|W|+2) \\
&= \frac{3}{2}\left(|W|-\frac{1}{2}\right)^2+\frac{1}{2}n^2-\frac{5}{2}n+\frac{13}{8}.
\end{align*}
Furthermore,  since $|W| \leq  (\frac{1}{3}-\gamma)n$ and $0 < \varepsilon \ll \gamma$,  we derive that
\begin{align*}
\deg(v_0)+\deg(u_1)  & \leq \frac{3}{2}\left(\frac{n}{3}-\gamma n-\frac{1}{2}\right)^2+\frac{1}{2}n^2-\frac{5}{2}n+\frac{13}{8}  =\left(\frac{2}{3}-\gamma+\frac{3}{2}\gamma^2\right)n^2-\left(3-\frac{3}{2}\gamma\right)n+2 \\
& < \frac{2}{3}n^2-\varepsilon n^2,
\end{align*}
contradicting our assumption on $\sigma'_2(H)$.

\medskip

\noindent{\bf Case 2:} $|W|> (\frac{1}{3}-\gamma)n$. 

\begin{claim}\label{claim3}
$|M| \geq n/3 - \gamma' n$. 
\end{claim}

\begin{proof}
To the contrary, assume that $|M| < n/3 - \gamma' n$. Fix $v_0\in W_2$. Then $\deg(v_0) \leq \binom{|U|}{2} - \binom{|U_2|}2$ because there is no edge of type $U_2 U_2 W_2$. Suppose $u\in U$ is adjacent to $v_0$. Trivially $\deg(u) \leq \binom{|U|-1}{2} + (|U| - 1)|W|$. Thus
\begin{align*}
\deg(v_0) + \deg(u) &\le \binom{|U|-1}{2} + (|U| - 1)|W| + \binom{|U|}{2} - \binom{|U_2|}2 = (n-1)(|U| - 1) -  \binom{|U_2|}2.
\end{align*}
Our assumptions imply that $|U|\leq 2n/3 + \gamma n$ and $|U_2|\geq2 \gamma'n$. As a result,
\[
\deg(v_0) + \deg(u) \leq \left(n-1\right)\left(\frac23 n + \gamma n -1 \right) - \binom{ 2\gamma'n}2< \frac{2}{3}n^2-\varepsilon n^2,
\]
because $\varepsilon \ll \gamma \ll \gamma'$ and $n$ is sufficiently large. This contradicts our assumption on $\sigma'_2(H)$.
\end{proof}

Fix $u_1 \ne u_2 \in U_2$ and $v_0 \in W_2$.
Trivially $\deg(w) \leq \binom{|U|}{2}$ for any vertex $w \in W$ and $\deg(u) \leq \binom{|U|-1}{2}+|W|(|U|-1)$ for any vertex $u \in U$. Furthermore, for any two distinct edges $e_1, e_2\in M$, we observe that at least one triple of type $UUW$ with one vertex from each of $e_1$ and $e_2$ and one vertex from $\{u_1,u_2,v_0\}$ is \emph{not} an edge -- otherwise there is a matching $M_3$ of size three on $e_1\cup e_2\cup \{u_1,u_2,v_0\}$ and $M_3\cup M\setminus \{e_1, e_2\}$ is thus a matching larger than $M$.
By Claim \ref{claim3}, $|M| \geq n/3 - \gamma'n$. Thus,

\begin{align*}
\deg(u_1)+ \deg(u_2)+ \deg(v_0)\leq 2\left( \binom{|U|-1}{2}+|W|(U|-1)\right) +\binom{|U|}{2}-\binom{n/3 - \gamma'n}{2}.
\end{align*}

On the other hand, since $|W| > (\frac{1}{3}-\gamma)n \ge n/4$, Claim  \ref{claim2} implies that $u_i$ is adjacent to some vertex in $W$ for $i=1,2$. We know that $v_0$ is adjacent to some vertex in $U$. Therefore, $\deg(u_i) > \left(2n^2/3-\varepsilon n^2\right)-\binom{|U|}{2}$ for $i=1,2$, and $\deg(v_0) > \left(2n^2/3-\varepsilon n^2\right)-\left(\binom{|U|-1}{2}+|W|(|U|-1)\right)$. It follows that
\begin{align*}
\deg(u_1)+ \deg(u_2)+\deg(v_0)> 3\left(\frac{2n^2}{3}-\varepsilon n^2\right) - 2\binom{|U|}{2}- \binom{|U|-1}{2}-|W|(|U|-1).
\end{align*}
The upper and lower bounds for $\deg(u_1)+ \deg(u_2)+\deg(v_0)$ together imply that
\begin{align*}
3\left( \binom{|U|-1}{2}+|W|(|U|-1)+\binom{|U|}{2}\right)-\binom{n/3 - \gamma'n}{2} &> 3\left(\frac{2n^2}{3}-\varepsilon n^2\right), \\
\text{or} \quad (|U| - 1)(n-1) - \frac13 \binom{n/3 - \gamma'n}{2} &> \frac{2n^2}{3}-\varepsilon n^2,
\end{align*}
which is impossible because $|U|\le 2n/3 + \gamma n$, $0 < \varepsilon \ll \gamma \ll \gamma' \ll 1$ and $n$ is sufficiently large.
This completes the proof of Lemma~\ref{lemma7}.

\section{Proof of Theorem \ref{theorem6}}

Choose a matching $M$ such that (i) $M$ covers all the vertices of $W$; (ii) subject to (i), $|M|$ is the largest. Lemma \ref{lemma7} implies that such a matching exists. Let $M_1=\{e\in M: e\cap W\neq \emptyset\}$, $M_2=M\setminus M_1$, and $U_3=V(H)\setminus V(M)$. We have $|M_1|=|W|$, $|M_2|=|M|-|W|$, $|U_3|=n-3|M|$.

Suppose to the contrary, that $|M| \leq n/3-1$.  Fix three vertices $u_1,u_2,u_3$ of $U_3$. We distinguish the following two cases.

\noindent{\bf Case 1:} $|M| \leq n/3-\eta n$.

Trivially there are at most $3|M|$ edges in $H$ containing $u_i$ and two vertices from the same edge of $M$ for $i = 1,2,3$. For any distinct $e_1$, $e_2$ from $M$, we claim that
\[\sum_{i=1}^3|L_{u_i}(e_1,e_2)| \leq 18.\]
Indeed, let $H_1$ be the 3-partite subgraph of $H$ induced on three parts $e_1$, $e_2$, and $\{u_1,u_2,u_3\}$.
We observe that $H_1$ does not contain a perfect matching -- otherwise, letting $M_1$ be a perfect matching of $H_1$, $\left(M\setminus \{e_1,e_2\}\right) \cup M_1$ is a  larger matching than $M$, a contradiction. Apply Lemma \ref{lemma1} with $n=k=s=3$, we obtain that $|E(H_1)| \leq 18$. Therefore $\sum_{i=1}^3|L_{u_i}(e_1,e_2)| \leq 18 $.

For any $e \in M_1$, we claim that
\[\sum_{i=1}^3|L_{u_i}(e,U_3)| \leq 6(|U_3|-1).\]
Indeed, assume $e =\{v_1,v_2,v_3\} \in M_1$ with $v_1 \in W$. Apply Lemma \ref{lemma2} with $A=\{u_1,u_2,u_3\}$, $V=U_3$, and $G_i=\left(U_3, L_{v_i}(U_3)\right)$ for $i=1,2,3$. Since $|M| \leq n/3-4$, we have $|B|=|U_3|-3 \geq 2$. By the maximality of $M$, no edge of $G_1$ is disjoint from an edge of $G_2$ or $G_3$. By Lemma \ref{lemma2}, $\sum_{i=1}^3\sum_{j=1}^3 \deg_{G_i}(u_j) \leq 6(|U_3|-1)$. Hence $\sum_{i=1}^3|L_{u_i}(e,U_3)| = \sum_{i=1}^3\sum_{j=1}^3 \deg_{G_i}(u_j)\le 6(|U_3|-1)$.

Similarly, for any $e \in M_2$,  we can apply Lemma \ref{lemma3} to obtain that \[\sum_{i=1}^3|L_{u_i}(e,U_3)| \leq 3(|U_3|+1).\]

Putting these bounds together gives
\begin{align*}
\sum_{i=1}^{3}\deg(u_i) &\leq 18\binom{|M|}{2}+9|M|+\sum_{i=1}^{3}|L_{u_i}(V(M_1),U_3)|+\sum_{i=1}^{3}|L_{u_i}(V(M_2),U_3)|\\
&\leq 18\binom{|M|}{2}+9|M|+6|M_1|(|U_3|-1)+3|M_2|(|U_3|+1).
\end{align*}
Since $|M_1|=|W|$, $|M_2|=|M|-|W|$, $|U_3|=n-3|M|$, we derive that
\begin{align}
\sum_{i=1}^{3}\deg(u_i) & \leq 18\binom{|M|}{2}+9|M|+6|W|(n-3|M|-1)+3(|M|-|W|)(n-3|M|+1)\nonumber\\
& =  (3n-9|W|+3)|M|+3|W|n-9|W|. \nonumber
\end{align}

Furthermore, $3n-9|W|+3 >0$ and $|M| \leq n/3-\eta n$ implies that

\begin{align}
\sum_{i=1}^{3}\deg(u_i) & \leq  (3n-9|W|+3)\left(\frac{n}{3}-\eta n\right)+3|W|n-9|W| \nonumber\\
& =   \left(9 \eta n-9\right)|W|+\left(1-3\eta\right)n^2+\left(1-3\eta\right)n. \label{1}
\end{align}

If $|W| \leq n/4$, from \eqref{1}, we have
\begin{align*}
\sum_{i=1}^{3}\deg(u_i) & \leq  \left(9 \eta n-9\right)\frac{n}{4}+\left(1-3\eta\right)n^2+\left(1-3\eta\right)n=\left(1-\frac{3}{4}\eta\right)n^2-\left(3\eta+\frac{5}{4}\right)n,
\end{align*}
which contradicts the condition $\sum_{i=1}^{3}\deg(u_i) \geq 3\left(\frac{n^2}{3}-\frac{\varepsilon n^2}{2}\right)$ because $u_i \in U_3$ for $i \in [3]$ and $\varepsilon \ll \eta$.

If $|W| > n/4$, Claim \ref{claim2} implies that $u_i$ is adjacent to one vertex of $W$,  $i=1,2,3$. Furthermore, $\deg(w) \leq \binom{|U|}{2}$ for $w \in W$.  So
\begin{align}
\sum_{i=1}^3\deg(u_i) > 3\left(\frac{2n^2}{3}-\varepsilon n^2 - \binom{|U|}{2} \right) = 3\left(\frac{2n^2}{3}-\varepsilon n^2 - \binom{n-|W|}{2} \right). \nonumber
\end{align}
The upper and lower bounds for $\sum_{i=1}^3\deg(u_i)$ together imply that
\begin{align*}
 \left(9 \eta n-9\right)|W|+\left(1-3\eta\right)n^2+\left(1-3\eta\right)n+3\binom{n-|W|}{2} & > 3\left(\frac{2n^2}{3}-\varepsilon n^2\right),
\end{align*}
which is a contradiction because $|W| > n/4$, $0< \varepsilon \ll \eta \ll 1$ and $n$ is sufficiently large.

\medskip
\noindent{\bf Case 2:} $|M| > n/3-\eta n$.

If $|M| = n/3-1$, then $|U_3| = 3$ and we can not apply Lemmas \ref{lemma2} and \ref{lemma3}. In fact, whenever $|M| > n/3-\eta n $, Lemma \ref{lemma1} suffices for our proof. 

Let $W'=\{v \in W: \deg(v) \leq (5/18+\tau)n^2\}$. Let $M'$ be the sub-matching of $M$ covering every vertex of $W'$. If $|W'| \leq \rho n$, we claim that $\deg_{H'}(u) \geq \left(\frac{5}{9}+\gamma\right)\binom{n}{2}$ for every vertex $u \in V(H')$, where $H': = H[V\setminus V(M')]$. Indeed, from the definition of $W'$, $\deg_{H}(u) > (5/18+\tau)n^2$ for every vertex $u \in V(H')$. Hence, $$\deg_{H'}(u) \geq \deg_H(u)-3n|W'| > \left(\frac{5}{18}+\tau \right)n^2-3n|W'|.$$ Since $|W'| \leq \rho n$, $0 < \gamma \ll \rho \ll \tau \ll 1$ and $n$ is sufficiently large, we have
\begin{align*}
\deg_{H'}(u)> \left(\frac{5}{18}+\tau\right)n^2-3\rho n^2> \left(\frac{5}{9}+\gamma\right)\binom{n}{2}.\\
\end{align*}
In addition, $n$ is divisible by $3$, so $|V(H')|$ is divisible by $3$.  \eqref{eq:Nu4} implies that $H'$ contains a perfect matching $M''$. Now $M' \cup M''$ is a perfect matching of $H$.

Therefore, we assume that $|W'| \geq \rho n$ in the rest of the proof. If one vertex of $u_1,u_2,u_3$, say, $u_1$, is adjacent to one vertex in $W'$, the definition of $W'$ implies that $\deg(u_1) > 2n^2/3- \varepsilon n^2-\left(\frac{5}{18}+\tau\right)n^2$. Recall that $\deg(u_i) > n^2/3-\varepsilon n^2/2$ for $i=2,3$. Thus 

\begin{align}\label{eq4}
\sum_{i=1}^3\deg(u_i) > \left(\frac{4}{3}n^2-2\varepsilon n^2\right)-\left(\frac{5}{18}+\tau\right)n^2= \left(\frac{19}{18}-2\varepsilon-\tau \right)n^2.
\end{align}

By Lemma \ref{lemma1}, we have

\begin{align*}
\sum_{i=1}^{3}\deg(u_i) &\leq 18\binom{|M|}{2}+9|M|+9|M|(n-3|M|-1)\\
& = -18\left(|M|-\frac{1}{4}n+\frac{1}{4}\right)^2+\frac{9}{8}n^2-\frac{9}{4}n+\frac{9}{8},
\end{align*}
where $18\binom{|M|}{2}$ accounts for edges between pairs of matching $M$, $9|M|$ for edges with two vertices in the same matching edge from $M$, and $9|M|(n-3|M|-1)$ for edges with one vertex in $V(M)$, one vertex in $U_3$. Since $|M| > n/3-\eta n$, it follows that  $$\sum_{i=1}^{3}\deg(u_i) \leq -18\left(\frac{n}{3}-\eta n-\frac{1}{4}n+\frac{1}{4}\right)^2+\frac{9}{8}n^2-\frac{9}{4}n+\frac{9}{8} = (1+3\eta-18\eta^2) n^2+(9\eta-3)n.$$ However, $(1+3\eta-18\eta^2) n^2+(9\eta-3)n < \left(\frac{19}{18}-2\varepsilon-\tau \right)n^2$ because  $0 < \varepsilon \ll \eta \ll \tau \ll 1$ and $n$ is sufficiently large. It contradicts \eqref{eq4}.

If none of these three vertices $u_1,u_2,u_3$ are adjacent to the vertices in $W'$, we have
\begin{align*}
\sum_{i=1}^{3}\deg(u_i) &\leq 18\binom{|M|-|M'|}{2}+9(|M|-|M'|)+9(|M|-|M'|)(n-3|M|-1)\\
& +3\binom{2|M'|}{2}+3(2|M'|)(n-3|M'|-1)\\
&= -3\left(|M'|+\frac{1}{2}n-\frac{3}{2}|M|\right)^2-\frac{45}{4}|M|^2+\frac{9}{2}n|M|-9|M|+\frac{3}{4}n^2.
\end{align*}
Here $18\binom{|M|-|M'|}{2}$ accounts for edges between pairs of matching $M\setminus M'$, $9(|M|-|M'|)$ for edges with two vertices in the same matching edge from $M\setminus M'$, $9(|M|-|M'|)(n-3|M|-1)$ for edges with one vertex in $V(M\setminus M')$, one vertex in $U_3$, $3\binom{2|M'|}{2}$ for edges with two vertices in $V(M')\setminus W'$, and $3(2|M'|)(n-3|M'|-1)$ for edges with one vertex in $V(M')\setminus W'$, one vertex in $V(H)\setminus V(M')$. Since $-n/2+3|M|/2 < 0$ and $|M'|=|W'| \geq \rho n$, then
\begin{align*}
\sum_{i=1}^{3}\deg(u_i) & \leq -3\left(\rho n+\frac{1}{2}n-\frac{3}{2}|M|\right)^2-\frac{45}{4}|M|^2+\frac{9}{2}n|M|-9|M|+\frac{3}{4}n^2\\
&=-18\left(|M|-\frac{1}{4}n-\frac{1}{4}\rho n+\frac{1}{4}\right)^2+\left(\frac{9}{8}-\frac{15}{8}\rho^2-\frac{3}{4}\rho\right)n^2-\frac{9}{4}\rho n-\frac{9}{4}n+\frac{9}{8}.\\
\end{align*}
Recall that $0 < \rho \ll 1$, so $\frac{1}{4}n+\frac{1}{4}\rho n-\frac{1}{4} <  \frac{n}{3}-\eta n$. Furthermore, $|M| > \frac{n}{3}-\eta n$, hence we have
\begin{align*}
\sum_{i=1}^{3}\deg(u_i) & \leq -18\left(\frac{n}{3}-\eta n-\frac{1}{4}n-\frac{1}{4}\rho n+\frac{1}{4}\right)^2+\left(\frac{9}{8}-\frac{15}{8}\rho^2-\frac{3}{4}\rho\right)n^2-\frac{9}{4}\rho n-\frac{9}{4}n+\frac{9}{8} \\ & = \left(1-3\rho^2-9\eta \rho+3\eta-18\eta^2\right)n^2+(9\eta-3)n,
\end{align*}
which contradicts the condition $\sum_{i=1}^3\deg(u_i) \geq 3 \left(n^2/3-\varepsilon n^2/2\right)$ because $0 < \varepsilon \ll \eta \ll \rho \ll 1$ and $n$ is sufficiently large. This completes the proof of Theorem~\ref{theorem6}.

\section{Concluding remarks}

In this paper we consider the minimum degree sum of two adjacent vertices that guarantees a perfect matching in $3$-graphs. Given $3\le k< n$ and $2\le s\le n/k$, can we generalize this problem to $k$-graphs not containing a matching of size $s$?
For $1\leq \ell \leq k$, let $H_{n,k,s}^\ell$ denote the $k$-graph whose vertex set is partitioned into two sets $S$ and $T$ of size $n-s\ell+1$ and $s\ell-1$, respectively, and whose edge set consists of all the $k$-sets with at least $\ell$ vertices in $T$. Apparently $H_{n,k,s}^\ell$ contains no matching of size $s$. A well-known conjecture of Erd\"{o}s ~\cite{Erd65} says that $H_{n,k,s}^1$  or $H_{n,k,s}^k$  is the densest $k$-graph on $n$ vertices not containing a matching of size $s$. It is reasonable to speculate that the largest $\sigma'_2(H)$ among all $k$-graphs $H$ on $n$ vertices not containing a matching of size $s$ is also attained by $H_{n,k,s}^\ell$.
Note that $H_{n,k,s}^k$ is a complete $k$-graph of order $sk-1$ together with $n-sk+1$ isolated vertices and thus $\sigma'_2(H_{n,k,s}^k)= 2\binom{sk-2}{k-1}$.  When $1 \leq \ell \leq k-2$, any two vertices of $H_{n,k,s}^\ell$ are adjacent and thus $\sigma'_2(H_{n,k,s}^\ell)=2\delta_1(H_{n,k,s}^\ell)$. When $\ell =k-1$, it is easy to see that $\sigma'_2(H_{n,k,s}^{k-1})=2\binom{s(k-1)-2}{k-1}+(n-s(k-1)+2)\binom{s(k-1)-2}{k-2}$.


Assume $s=n/k$. Since $H_{n,k,n/k}^k$ contains isolated vertices and $\delta_1(H_{n,k,n/k}^\ell)\le \delta_1(H_{n,k,n/k}^1)$ for $1\le \ell \le k-2$, we only need to compare $\sigma'_2(H_{n,k,n/k}^{1})$ and $\sigma'_2(H_{n,k,n/k}^{k-1})$. For sufficiently large $n$, it is easy to see that $\sigma'_2(H_{n,k,n/k}^{1}) < \sigma'_2(H_{n,k,n/k}^{k-1})$ when $k \leq 6$ and $\sigma'_2(H_{n,k,n/k}^{1}) > \sigma'_2(H_{n,k,n/k}^{k-1})$ when $k \geq 7$.

\begin{problem} Does the following hold for any sufficiently large $n$ that is divisible by $k$? Let $H$ be a $k$-graph of order $n$ without isolated vertex. If $k \leq 6$ and $\sigma_2'(H) > \sigma'_2(H_{n,k,n/k}^{k-1})$ or $k \geq 7$ and $\sigma_2'(H) > \sigma'_2(H_{n,k,n/k}^1)$, then $H$ contains a perfect matching.
\end{problem}

Now assume $k=3$ and $2\le s\le n/3$. Note that
\begin{align*}
\sigma'_2(H_{n,3,s}^3) &= 2 \binom{3s-2}{2}, \quad
\sigma'_2(H_{n,3,s}^1) =2\left(\binom{n-1}{2}-\binom{n-s}{2} \right), \ \text{and} \\
\sigma'_2(H_{n,3,s}^2) & =\binom{2s-2}{2}+\left(n-2s+1\right)\binom{2s-2}{1}+\binom{2s-1}{2} = (2s-2)(n-1).
\end{align*}
It is easy to see that $\sigma'_2(H_{n,3,s}^2) > \sigma'_2(H_{n,3,s}^1)$. Zhang and Lu \cite{Yi2} made the following conjecture.

\begin{conjecture}\cite{Yi2}\label{con2}
There exists $n_0 \in \mathbb{N}$ such that the following holds. Suppose that $H$ is a $3$-graph of order $n \geq n_0$ without isolated vertex. If $\sigma_2'(H) > 2\left( \binom{n-1}{2}-\binom{n-s}{2}\right)$ and $n \geq 3s$, then $H$ contains no matching of size $s$ if and only if $H$ is a subgraph of $H_{n,3,s}^2$.
\end{conjecture}
Zhang and Lu~\cite{Yi2} showed that the conjecture holds when $n \geq 9s^2$. Later the same authors \cite{Yi3} proved the conjecture for $n \geq 13s$. If Conjecture \ref{con2} is true, then it implies the following theorem of K\"{u}hn, Osthus and Treglown \cite{Kuhn2}.

\begin{theorem}\cite{Kuhn2}
There exists $n_0 \in \mathbb{N}$ such that if $H$ is a $3$-graph of order $n \geq n_0$ with $\delta_1(H) \geq \binom{n-1}{2}-\binom{n-s}{2}+1$ and $n \geq 3s$, then $H$ contains a matching of size $s$.
\end{theorem}

Our Theorem~\ref{the1} suggests a weaker conjecture than Conjecture \ref{con2}.

\begin{conjecture}\label{con3}
There exists $n_1 \in \mathbb{N}$ such that the following holds. Suppose that $H$ is a $3$-graph of order $n \geq n_1$ without isolated vertex. If $\sigma_2'(H) > \sigma'_2(H_{n,3,s}^2)$ and $n \geq 3s$, then $H$ contains a matching of size $s$.
\end{conjecture}

On the other hand, we may allow a $3$-graph to contain isolated vertices. Note that  $\sigma'_2(H_{n,3,s}^2) \geq \sigma'_2(H_{n,3,s}^3)$ if and only if $s \le (2n+4)/9$. 
We make the following conjecture.
\begin{conjecture}\label{con4}
There exists $n_2 \in \mathbb{N}$ such that the following holds. Suppose that $H$ is a $3$-graph of order $n \geq n_2$ and $2\le s\le n/3$. If $\sigma_2'(H) > \sigma'_2(H_{n,3,s}^2)$ and $s \leq (2n+4)/9$ or $\sigma_2'(H) > \sigma'_2(H_{n,3,s}^3)$ and $s > (2n+4)/9$, then $H$ contains a matching of size $s$.
\end{conjecture}

In fact, we can derive Conjecture \ref{con4} from Conjecture \ref{con3} as follows. Let $n_2 = \max \{ \binom{n_1}{2}, \frac{3}{2} n_1 \}$ and $H$ be a $3$-graph of order $n \geq n_2$ satisfying the assumption of Conjecture \ref{con4}. If $H$ contains no isolated vertex, then $H$ contains a matching of size $s$ by Conjecture \ref{con3}. Otherwise, let $W$ be the set of isolated vertices in $H$. Let $H' = H[V(H)\setminus W']$ and $n' = n - |W|$. Then $H'$ is a 3-graph without isolated vertex and $\sigma_2'(H') = \sigma_2'(H)$. When $2\le s\le (2n+4)/9$, we have $\sigma_2'(H') > \sigma'_2(H_{n,3,s}^2) > \sigma'_2(H_{n',3,s}^2)$. In addition, since $n\ge \binom{n_1}2$ and
\[
2\binom{n' - 1}2 \ge \sigma'_2(H')> (2s-2)(n-1)\ge 2(n-1),
\]
we have $n'\ge n_1$. When $s> (2n+4)/9$, we have $\sigma_2'(H') > \sigma'_2(H_{n,3,s}^3) > \sigma'_2(H_{n,3,s}^2) > \sigma'_2(H_{n',3,s}^2)$. In addition, since $n\ge 3n_1/2$ and
\[
2\binom{n' - 1}2 \ge \sigma'_2(H')> 2\binom{3s-2}2> 2\binom{2(n-1)/3}2,
\]
we have $n'\ge n_1$. In both cases, Conjecture \ref{con3} implies that $H'$ contains a matching of size $s$.

%


\end{document}